\documentclass[12pt]{article}
\usepackage[latin1]{inputenc}
\usepackage{amsmath}
\usepackage{amsfonts}
\usepackage{amssymb}
\usepackage{color}
\usepackage{graphicx}
\newtheorem{theorem}{Theorem}[section]

\newcommand{\qed}{\hspace{\stretch{3}}$\square$\\[1.8ex]}

\newtheorem{lemma}[theorem]{Lemma}
\newtheorem{cor}[theorem]{Corollary}

\newtheorem{teo}[theorem]{Theorem}
\newtheorem{remark}[theorem]{Remark}

\newtheorem{definition}[theorem]{Definition}

\def\R{\mathbb{R}}

\def\R{\mathbb{R}}

\title{\bf  The backbone decomposition for spatially dependent supercritical superprocesses}
\author{{\sc A.E. Kyprianou\footnote{Department of Mathematical Sciences, University of Bath,
Claverton Down, Bath, BA2 7AY, U.K.} \ \
J-L. P\'erez\footnote{Department of Statistics, ITAM, Rio Hondo 1,
Tizapan 1 San Angel, 01000 M\'exico, D.F.} \ \
  Y.-X.
Ren\footnote{LMAM School of Mathematical Sciences $\&$ Center for Statistical Science,
Peking University, Beijing 100871,
 P. R. China.}
}
}

\begin{document}

\maketitle
\begin{abstract}  Consider any supercritical Galton-Watson process which may become extinct with positive probability. It is a well-understood and intuitively obvious phenomenon that,
 on the survival  set, the  process may be pathwise decomposed into a stochastically `thinner' Galton-Watson process, which almost surely survives and which is decorated with immigrants, at every time step, initiating independent copies of the original Galton-Watson process conditioned to become extinct. The thinner process is known as the {\it  backbone} and characterizes the  genealogical lines of  descent of prolific individuals in the original process. Here, prolific means individuals who  have  at least one descendant in every subsequent generation to their own.

Starting with Evans and O'Connell \cite{EO}, there exists a cluster of literature, \cite{EP, SV, DW, BKMS, KR},  describing  the analogue of this decomposition (the so-called {\it backbone decomposition}) for a variety of different classes of superprocesses and continuous-state branching processes. Note that the latter family of stochastic processes may be seen as the total mass process of superprocesses with non-spatially dependent branching mechanism.

In this article we consolidate  the aforementioned collection of results concerning backbone decompositions     and describe a  result for a general class of supercritical superprocesses  with spatially dependent branching mechanisms. Our approach exposes the commonality and robustness of many of the existing arguments in the literature.

\bigskip


\bigskip


\end{abstract}
\section{Superprocesses and Markov branching processes}

This paper concerns a fundamental decomposition which can be found amongst a  general family of  superprocesses and has, to date, been identified for a number of specific sub-families thereof by a variety of different authors. We therefore start by briefly describing   the general family of superprocesses that we shall concern ourselves with. The reader is referred to the many, and now classical, works of Dynkin for further details of what we present below; see for example \cite{D0, D1, D3, D4, D2}. The books of Le Gall \cite{LeG}, Etheridge \cite{Etheridge} and Li \cite{ZL} also serve as an excellent point of reference.

Let $E$ be a domain of $\mathbb{R}^d$.
Following the setting of Fitzsimmons \cite{fitz}, we are interested in strong Markov processes, $X = \{X_t : t\geq 0\}$ which are valued in $\mathcal{M}_F(E)$, the space of finite measures with support in $E$. The evolution of $X$ depends on two quantities $\mathcal{P}$ and $\psi$. Here, $\mathcal{P}=\{\mathcal{P}_t:t\geq0\}$ is the semi-group of a   diffusion on $E$  {\color{black}(the diffusion  killed  upon leaving $E$)}, and $\psi$ is a so-called branching mechanism which, by assumption, takes the form
\begin{equation}\label{bm}
\psi(x,\lambda)=-\alpha(x)\lambda +\beta(x)\lambda^2+\int_{(0,\infty)}(e^{-\lambda z}-1+\lambda z)\pi(x,{\rm d}z),
\end{equation}
where $\alpha$ and $\beta\geq 0$ 
are  bounded measurable mappings from $ E$ to $\mathbb{R}$ and $[0,\infty)$ respectively and for each $x\in E$, $\pi(x, {\rm d}z)$ is a measure concentrated  on  $(0,\infty)$ such that $x\to\int_{(0,\infty)}(z\wedge z^2)\pi(x,{\rm d}z)$ is bounded and measurable. For technical reasons, we shall additionally assume that the diffusion associated to $\mathcal{P}$ satisfies certain conditions. These conditions are lifted from Section II.1.1. (Assumptions 1.1A and 1.1B) on pages 1218-1219 of \cite{D1}\footnote{The assumptions on $\mathcal{P}$ may in principle be relaxed. The main reason for this imposition here comes in the proof of Lemma \ref{limit-Delta-Z} where a comparison principle is used for diffusions.}. They state that $\mathcal{P}$ has associated infinitesimal generator
\[
L= \sum_{i,j}a_{i,j}\frac{\partial^2}{\partial x_i\partial x_j} + \sum_i b_i \frac{\partial}{\partial x_i},
\]
where the coefficients $a_{i,j}$ and $b_j$ are space dependent coefficients satisfying:

\medskip

\noindent{\bf (Uniform Elliptically)} There exists a constant $\gamma>0$ such that
\[
\sum_{i,l}a_{i,j}u_i u_j\geq\gamma\sum_i u_i^2
\]
for all $x\in E$ and $u_1, \cdots u_d\in \mathbb{R}.$

\medskip

\noindent{\bf (H\"older continuity)} The coefficients $a_{i,j}$ and $b_i$ are uniformly bounded and H\"older continuous in such way  that there exist a constants $C>0$ and $\alpha\in(0,1]$ with
\[
|a_{i,j}(x) - a_{i,j}(y)|,\quad |b_i(x) -b_i(y)|\leq C |x-y|^\alpha
\]
 for all $x,y\in E$.
 Throughout, we shall refer to  $X$ as the $(\mathcal{P}, \psi)$-superprocess.

For each $\mu\in\mathcal{M}_F({ E})$ we denote by $\mathbb{P}_\mu$ the law of $X$ when issued from initial state $X_0 = \mu$. The semi-group of $X$, which in particular characterizes the laws $\{\mathbb{P}_\mu: \mu\in\mathcal{M}_F({ E})\}$, can be described as follows.
 For each $\mu\in\mathcal{M}_F({ E})$ and all $f\in {\rm bp}({ E})$, the space of non-negative, bounded measurable functions on ${ E}$,
\begin{equation}
\mathbb{ E}_\mu(e^{-\langle f, X_t\rangle})=\exp\left\{-\int_{{ E}}u_f(x,t)\mu({\rm d} x)\right\} \qquad t\geq 0,
\label{sg-fixed-time}
\end{equation}
where $u_f(x,t)$ is the unique non-negative solution to the equation
\begin{equation}
\label{u_f}
u_f(x,t)=\mathcal{P}_t[f](x)-\int^{t}_0{\rm d}s\cdot  \mathcal{P}_s[\psi(\cdot, u_f(\cdot, t-s))](x)\qquad x\in{ E}, t\geq 0.
\end{equation}
See for example Theorem 1.1 on pages 1208-1209 of \cite{D1} or Proposition 2.3 of \cite{fitz}.
Here we have used the standard inner product notation,
\[
\langle f , \mu\rangle =  \int_{{ E} }f(x)\mu({\rm d}x),
\]
for   $\mu\in\mathcal{M}_F({ E})$ and any $f$ such that the integral makes sense.

Suppose that we define $\mathcal{E} = \{\langle 1, X_t\rangle =  0\mbox{ for some }t>0\}$, the event of {\it  extinction}.
For each $x\in{ E}$ write
\begin{equation}
w(x)=-\log \mathbb{P}_{\delta_x}({\cal E}).
\label{w-def}
\end{equation}
It follows from \eqref{sg-fixed-time} that
\begin{equation}
\mathbb{ E}_\mu(e^{-\theta\langle 1, X_t\rangle})=\exp\left\{-\int_{{ E}}u_\theta(x,t)\mu({\rm d} x)\right\} \qquad t\geq 0,
\end{equation}
Note that $u_\theta(t,x)$ is increasing in $\theta$ and that $\mathbb{P}_\mu(\langle 1, X_t\rangle=0)$ is monotone increasing. Using these facts and letting $\theta\to\infty$, then $t\to\infty$,  we get that
\begin{equation}\label{extinct-t}
\mathbb{P}_\mu({\cal E}) =\lim_{t\to\infty}\mathbb{P}_\mu(\langle 1, X_t\rangle=0)=\exp\left\{-\int_{{ E}}\lim_{t\to\infty}\lim_{\theta\to\infty}u_{\theta}(x,t)\mu({\rm d} x)\right\}.
\end{equation}
 By choosing $\mu = \delta_x$, with $x\in { E}$, we see that
\begin{equation}\label{assumeforeextinguishing}
\mathbb{P}_\mu({\cal E})=\exp\left\{-\int_{{ E}}w(x)\mu({\rm d} x)\right\}.
\end{equation}

For the special case that $\psi$ does not depend on $x$ and $\mathcal{P}$ is conservative,  $\langle 1, X_t\rangle$ is a continuous state branching process. If $\psi(\lambda)$ satisfy the following condition:
$$\int^\infty\frac{1}{\psi(\lambda)}{\rm d}\lambda<\infty,$$
then $P_\mu$ almost surely we have ${\cal E}=\{\lim_{t\to\infty}\langle 1, X_t\rangle=0\}$, that is to say the event of {\it extinction} is equivalent to the event of {\it extinguishing}, see \cite{BKMS} and \cite{KR} for examples.

By first conditioning the event $\mathcal{ E}$ on $\mathcal{F}_t : =\sigma\{X_s : s\leq t\}$, we find that for all $t\geq 0$,
\[
\mathbb{ E}_\mu (e^{-\langle w, X_t\rangle}) = e^{-\langle w, \mu\rangle}.
\]
The function $w$ will play an important role in the forthcoming analysis and henceforth we shall assume that it respects the following property.

\bigskip

{\bf (A):} $w$ is locally bounded away from $0$ and $\infty$. 
\bigskip

The pathwise evolution of superprocesses is somewhat difficult to visualise on account of their realisations at each fixed time being sampled from the space of finite measures. However a related class of stochastic processes which exhibit similar mathematical properties to superprocesses and whose paths are much easier to visualise is that of  Markov branching processes. A Markov branching process $Z = \{Z_t: t\geq 0\}$  takes values in the space  $\mathcal{M}_a({ E})$ of finite atomic measures in ${ E}$ taking the form $\sum_{i=1}^n\delta_{x_i}$, where $n\in\mathbb{N}\cup\{0\}$ and $x_1,\cdots,x_n\in{ E}$. To describe its evolution we need to specify two quantities, $(\mathcal{P}, F)$, where, as before, $\mathcal{P}$ is  the semi-group of a diffusion on ${ E}$ and $F$ is the so-called branching generator which takes the form
\begin{equation}
F(x, s) =  q(x)\sum_{n\geq 0}p_n(x)(s^n - s), \qquad x\in{ E}, s\in[0,1],
\label{branch-gen}
\end{equation}
where $q$ is a bounded measurable mapping from ${ E}$ to $[0,\infty)$ and, the measurable  sequences $\{p_n(x): n\geq 0\}$, $x\in E$, are probability 
distributions. For each $\nu\in\mathcal{M}_a({ E})$, we denote by ${\rm P}_\nu$ the law of $Z$ when issued from initial state $Z_0 = \nu$. The probability ${\rm P}_\nu$ can be constructed in a pathwise sense as follows. From each point in the support of $\nu$ we issue an independent copy of
the diffusion
with semi-group $\mathcal{P}$. Independently of one another, for $(x,t)\in{ E}\times[0,\infty)$, each of these particles will be  killed at rate $q(x){\rm d}t$  to be replaced at their space-time point of death by $n\geq 0$ particles with probability $p_n(x)$. Relative to their point of creation, new particles behave independently to one another, as well as to existing particles, and undergo the same life cycle in law as their parents.

By conditioning on the first split time in the above description of a $(\mathcal{P}, F)$-Markov branching process, it is also possible to show that
\[
{\rm E}_\nu(e^{-\langle f, Z_t\rangle})=\exp\left\{-\int_{{ E}}v_f(x,t)\nu({\rm d} x)\right\} \qquad t\geq 0,
\]
where $v_f(x,t)$ solves 
\begin{equation}
\label{v_f}
e^{-v_f(x,t)}=\mathcal{P}[e^{-f}](x)+\int^{t}_0{\rm d}s\cdot  \mathcal{P}_s[F(\cdot, e^{-v_f(\cdot, t-s)})](x)\qquad x\in{ E}, t\geq 0.
\end{equation}
Moreover, it is known, cf.  Theorem 1.1 on pages 1208-1209 of \cite{D1}, that the solution to this equation is unique.
This shows a similar characterisation of the semi-groups of Markov branching processes to those of superprocesses.

The close similarities between the two processes become clearer when one takes account of the fact that the existence of superprocesses can be justified through a high density scaling  procedure of Markov branching processes. Roughly speaking, for a fixed triplet, $\mu,\mathcal{P},\psi$, one may construct a sequence of Markov branching processes, say $\{Z^{(n)}:n\geq 1\}$, such that the $n$-th element of the sequence is issued with an initial configuration of points which is taken to be an independent Poisson random measure with intensity $n\mu$ and branching generator $F_n$ satisfying
\[
F_n(x,s) = \frac{1}{n}[\psi(x,n(1-s)) + \alpha(x)n(1-s)],\qquad x\in{ E}, s\in[0,1].
\]
It is not immediately obvious that the right-hand side above conforms to the required structure of branching generators as stipulated in (\ref{branch-gen}), however this can be shown to be the case; see for example the discussion on p.93 of \cite{ZL}. It is now a straightforward exercise to show that for all $f\in {\rm bp}({ E})$ and $t\geq 0$ the law of $\langle f,n^{-1}Z^{(n)}_t \rangle$ converges weakly to the law of $\langle f, X_t\rangle$, where the measure $X_t$ satisfies (\ref{sg-fixed-time}). A little more work shows the convergence of the sequence of processes $\{n^{-1}Z^{(n)}: n\geq 1\}$ in an appropriate sense to a $(\mathcal{P}, \psi)$-superprocess issued from an initial state $\mu$.

Rather than going into the details of this scaling limit, we focus instead in this paper on another connection between superprocesses and branching processes which explains their many similarities without the need to refer to  a scaling limit.  The basic idea is that, under suitable assumptions, for a given $(\mathcal{P}, \psi)$-superprocess, there exists a related Markov branching process, $Z$, with computable characteristics such that at each fixed $t\geq 0$, the law of $Z_t$ may be coupled to the law of $X_t$ in such a way that, given $X_t$, $Z_t$ has the law of a Poisson random measure with intensity $w(x)X_t({\rm d}x)$, where $w$ is given by (\ref{w-def}). The study of so-called {\it backbone decompositions} pertains to how the
aforementioned Poisson embedding
  may be implemented in a pathwise sense at the level of processes.

The remainder of this paper is structured as follows. In the next section we briefly review the  sense and settings in which backbone decompositions have been  previously studied. Section \ref{prelim} looks at some preliminary results needed to address the general backbone decomposition that we deal with in Sections \ref{localbb}, \ref{5} and \ref{globalbb}.

\section{A brief history of backbones}

The basic idea of a backbone decomposition can be traced back to the setting of Galton-Watson trees with ideas coming from Harris and Sevast'yanov; cf Harris \cite{Harris}. Within any supercritical Galton-Watson process with a single initial ancestor  for which the probability of survival is not equal  to 0 or  1, one may identify prolific genealogical lines of descent on the event of survival. That is to say, infinite sequences of descendants which have the property that every individual has at least one descendant in every subsequent generation beyond its own. Together, these prolific genealogical lines of descent make a Galton-Watson tree which is thinner than the original tree.  One may describe the original Galton-Watson process in terms of  this thinner Galton-Watson process, which we now refer to as a  {\it backbone}, as follows. Let $0<p<1$ be the probability of survival. Consider a branching process which, with probability $1-p$, is an independent copy of the original Galton-Watson process conditioned to become extinct and, with probability $p$, is a copy of the backbone process, having the additional feature that every individual in the backbone process immigrates an additional random number of offspring, each of which initiate independent copies of the original Galton-Watson process conditioned to become extinct. With an appropriate choice of immigration numbers,  the resulting object has the same law as the original Galton-Watson process.

In Evans and O'Connell \cite{EO}, and later in Engl\"ander and Pinsky \cite{EP}, a new decomposition of a supercritical superprocess with quadratic branching mechanism was introduced in which one may write the distribution of the superprocess at time $t\geq0$ as the result of summing two independent processes together. The first is a copy of the original process conditioned
on extinction.
The second process is understood as the superposition of mass from independent copies of the original process conditioned on extinction which have immigrated `continuously' along the path of an auxiliary dyadic branching particle diffusion which starts with a random number of initial ancestors whose cardinality and spatial position is governed by an independent Poisson point process. The embedded branching particle system is known as the {\it backbone} (as opposed to the {\it spine} or {\it immortal particle} which appears in another related decomposition,   introduced in Roelly-Coppoletta and Rouault \cite{RR} and Evans  \cite{E}). In both \cite{EO} and \cite{EP} the decomposition is seen through the semi-group evolution equations which drive the process semi-group. However no pathwise construction is offered.

A pathwise backbone decomposition appears in Salisbury and Verzani \cite{SV}, who consider the case of conditioning a super-Brownian motion as it exits  a given  domain such that the exit measure contains at least $n$ pre-specified points in its support. There it was found that the conditioned process has the same law as the superposition of mass that immigrates in a Poissonian way along the spatial path of a branching particle motion which exits the domain with precisely $n$ particles at the pre-specified points. Another pathwise backbone decomposition for branching particle systems is given in Etheridge and Williams \cite{EW}, which is used in combination with a limiting procedure to prove another version of Evan's immortal particle picture.

In Duquesne and Winkel \cite{DW} a version of the Evans-O'Connell  backbone  decomposition was established for more general branching mechanisms, albeit without taking  account of spatial motion. In their case, quadratic branching is replaced by a general branching mechanism $\psi$ which is the Laplace exponent of a spectrally positive L\'evy process and which satisfies the conditions
 $0<-\psi'(0+)<\infty$ and
 $
\int^\infty   1/\psi(\xi)  {\rm d}\xi <\infty.
$
Moreover, the decomposition is offered in the pathwise sense and is described through the growth of genealogical trees embedded within the underling continuous state branching process.
The backbone is a continuous-time Galton Watson process and the general nature of the branching mechanism induces three different kinds of immigration. Firstly there is continuous immigration which  is described by a Poisson point process of independent processes along the trajectory of the backbone where the rate of immigration is given by a so-called excursion measure which assigns zero initial mass,  and finite life length of the  immigrating processes. A second Poisson point process  along the backbone describes the immigration of independent processes where the rate of immigration is given by the law of the original process conditioned on extinguishing and with a positive initial  volume of mass randomised by an infinite measure. This accounts for so-called discontinuous immigration. Finally, at the times of branching of the backbone, independent copies of the original process conditioned on extinguishing are immigrated with randomly distributed initial mass which depends on the number of offspring at the branch point. The last two forms of immigration do not occur when the branching mechanism is purely quadratic.

Concurrently to the work of \cite{DW} and within the  class of branching mechanisms corresponding to spectrally positive L\'evy processes with paths of
 unbounded variation (also allowing for the case that $-\psi'(0+)=\infty$), Bertoin et al. \cite{Betal} identify the aforementioned backbone
 as characterizing prolific genealogies within  the underling continuous state branching process.

 Berestycki et al. \cite{BKMS} extend the results of \cite{EO} and \cite{DW}, showing that for  superprocesses with relatively general motion and non-spatial branching mechanism corresponding to spectrally positive L\'evy  processes with finite mean, a pathwise backbone decomposition arises. The role of the backbone is played by a branching particle diffusion with the same motion operator as the superprocesses and, like Salisbury and Verzani
\cite{SV}, additional mass immigrates along the trajectory of the backbone in a Poissonian way.  Finally Kyprianou and Ren \cite{KR} look at the case of a continuous-state branching process with immigration for which a similar backbone decomposition to \cite{BKMS} can be shown.

As alluded to in the abstract, our objective in this article is to provide a general backbone decomposition which overlaps  with many of the cases listed above and, in particular, exposes the general effect on the backbone of spatially dependent branching. It is also our intention to demonstrate the robustness of some of the arguments that have been used in earlier work on  backbone decompositions. Specifically we are referring to the original manipulations associated with the semi-group equations given in Evans and O'Connell \cite{EO} and Engl\"ander and Pinsky \cite{EP}, as well as the use the Dynkin-Kuznetsov excursion measure, as found in Salisbury and Verzani \cite{SV}, Berestyki et al. \cite{BKMS} and Kyprianou and Ren \cite{KR}, to describe the rate of immigration along the backbone.

\section{Preliminaries}\label{prelim}

Before stating and proving the backbone decomposition, it will first be necessary to describe a number of mathematical structures which will play an important role.

\subsection{Localisation}

Suppose that the stochastic process $\xi = \{\xi_t : t\geq 0\}$ on $ E\cup\{\dagger\}$, where $\dagger$ is its cemetery state, is the
diffusion in $E$ corresponding to the semi-group $\mathcal{P}$. We shall denote its  probabilities by $\{\Pi_x : x\in { E}\}$. In the next definition, we shall take ${\rm bp}(E\times[0,t] )$ to be the space of non-negative, bounded measurable functions on $ E\times[0,t]$ with the additional property that the value of $f(x,s)$ on $ E\times[0,t]$ is independent of $s$,  and it is implicitly understood that  for all functions $f\in{\rm bp}( E\times[0,t] )$, we extend their spatial domain to include $\{\dagger\}$ and set $f(\{\dagger, s\}) = 0$.

\begin{definition}\rm \label{Laplace-1} For any open, bounded set {\color{black} $D\subset\subset E$}, and $t\geq 0$,
there exists a random measure $\widetilde X^D_t$  supported on
the boundary of $D\times[0,t)$ such that,
for all $f\in {\rm bp}({\color{black} E}\times[0,t] )$
and $\mu \in {\cal M}_F(D)$, the space of finite measures on $D$,
\begin{equation}\label{Laplace-D}-\log \mathbb{ E}_{\mu}\left(e^{-\langle f, \widetilde X^{D}_t\rangle}\right)=\int_{ E}\widetilde u^D_f(x, t)\mu({\rm d}x),\quad \quad t\ge 0,\end{equation}
where $\widetilde u_f^D(x,t)$ is the unique non-negative solution to the integral equation
\begin{equation}\label{int-D}\widetilde u^D_f(x,t)=\Pi_x[f(\xi_{t\wedge\tau^D}, t\wedge\tau^D)]-\Pi_x\left[\int^{t\wedge\tau^D}_0\psi(\xi_s, \widetilde u^D_f(\xi_s, t-s)){\rm d}s\right],
\end{equation}
and $\tau^D=\inf\{t\ge 0, \xi_t\in D^c\}$. Note that, here, we use the obvious notation that
$\langle f, \widetilde X^D_t\rangle  = \int_{\partial(D\times [0,t))} f(x,s) \widetilde X_t^D({\rm d}x, {\rm d}s)$.
Moreover, with a slight abuse of notation, since their effective spatial domain is
restricted to $D\cup\{\dagger\}$
in the above equation, we treat $\psi$ and $\widetilde{u}^D_f$ as functions in ${\rm bp}({ E}\times[0,t])$ and accordingly it is clear how to handle a spatial argument equal to $\dagger$, as before.
In the language of Dynkin \cite{D4}, $\widetilde X^D_t$ is called an exit measure.

 Now we define a random measure $ X^D_t$ on $D$ such that  $\langle  f,  X^D_t\rangle=\langle f, \widetilde X^{D}_t\rangle$ for any $f\in {\rm bp}(D)$,
 the space of non-negative, bounded measurable functions on $D$, where, henceforth, as is appropriate, we regard $f$ as a function defined on
  ${\color{black}{ E}\times [0,\infty)}$ in the sense that
 \begin{equation}
 f(x,t)= \left\{\begin{array}{ll}f(x),&\, x\in D\\
 0,&\, x\in E\setminus D.\end{array}\right.
 \label{lid}
 \end{equation} Then for any $f\in{\rm bp}(D)$ and $\mu \in {\cal M}_F(D)$,
\begin{equation}\label{Laplace-D-r}-\log \mathbb{ E}_{\mu}\left(e^{-\langle f,  X^{D}_t\rangle}\right)=\int_{ E}u^D_f(x, t)\mu({\rm d}x),\quad \quad t\ge 0,\end{equation}
where $u_f^D(x,t)$ is the unique non-negative solution to the integral equation
\begin{equation}\label{int-D'}u^D_f(x,t)=\Pi_x[f(\xi_{t}); t<\tau^D]-\Pi_x\left[\int^{t\wedge\tau^D}_0\psi(\xi_s, u^D_f(\xi_s,  t-s)){\rm d}s\right],
\quad x\in D.\end{equation}
\end{definition}

As a process in time, $\widetilde X^D = \{\widetilde X_t^D: t\geq 0\}$ is a superprocess with branching mechanism
$\psi(x,\lambda)\mathbf{1}_D(x)$, but whose associated semi-group is replaced by that of the process $\xi$
absorbed on $\partial D$.
Similarly, as a process in time, $X^D = \{X_t^D: t\geq 0\}$ is a superprocess with branching mechanism
$\psi(x,\lambda)\mathbf{1}_D(x)$, but whose associated semi-group is replaced by that of the process $\xi$
killed upon leaving  $D$.
 One may  think of $X^D_t$ as describing the mass at time $t$ in $X$ which {\it historically} avoids exiting the domain $D$. Note moreover that  for any two open bounded domains, $D_1\subset\subset D_2\subset\subset E $,  the processes $\widetilde X^{D_1}$ and $\widetilde X^{D_2}$ (and hence $X^{D_1}$ and $X^{D_2}$) are consistent in the sense that
 \begin{equation}
 \widetilde X^{D_1}_t= (\widetilde{\widetilde X_t^{D_2}})^{D_1},
 \label{consistent}
 \end{equation}
  for all $t\geq 0$ (and similarly $X_t^{D_1} = (X_t^{D_2})^{D_1}$ for all $t\geq 0$).

\subsection{Conditioning on
extinction}

In the spirit of the relationship between (\ref{Laplace-D}) and  (\ref{int-D}), we have that $w$ is the unique solution to
\begin{equation}\label{int-w-D}
w(x)=\Pi_x[w(\xi_{t\wedge \tau^D})]-\Pi_x\left[\int_0^{t\wedge \tau^D}\psi(\xi_s,w(\xi_s)){\rm d}s\right], \qquad x\in{D}.
\end{equation}
for all open domains {\color{black}$D\subset\subset E$}. {\color{black} Again, with a slight abuse of notation, we treat $w$ with its spatial domain
$ E\cup\{\dagger\}$
as a function in ${\rm bp}({ E}\times [0,t])$ and  $w(\dagger) : =0$}.
From Lemma 1.5 in \cite{D1} we may  transform (\ref{int-w-D}) to the equation
\[
w(x) = \Pi_x \left[ w(\xi_{t\wedge\tau_D})\exp\left\{ - \int_0^{t\wedge\tau_D}\frac{\psi(\xi_s,w(\xi_s))}{w(\xi_s)}{\rm d}s\right\}\right],\qquad x\in{D},
\]
which shows that for all open bounded domains $D$,
\begin{equation}
w(\xi_{t\wedge \tau^D}) \exp\left\{ - \int_0^{t\wedge\tau^D}\frac{\psi(\xi_s,w(\xi_s))}{w(\xi_s)}{\rm d}s\right\}, \qquad t\geq 0,
\label{localmg}
\end{equation}
is a martingale.

The function $w$ can be used to locally describe the law of the superprocess when conditioned on {\it global}
{\it extinction} (as opposed to extinction on the sub-domain $D$).
The following lemma outlines  standard theory.

\begin{lemma}\label{lemma2} Suppose that $\mu\in \mathcal{M}_F(E)$ satisfies $\langle w,\mu\rangle<\infty$ (so, for example, it suffices that $\mu$ is compactly supported). Define
$$\mathbb{P}^*_\mu(\cdot)=\mathbb{P_\mu}(\cdot|{\cal E}).$$
Then
for any $f\in {\rm bp}{\color{black}( { E}\times[0,t])} $ with the additional property that the value of {\color{black} $f(x,s)$ on $ E\times[0,t]$} is independent of $s$ and $\mu \in {\cal M}_F(D)$,
$$-\log \mathbb{ E}^*_{\mu}\left(e^{-\langle f, \widetilde X^D_t\rangle}\right)=\int_{{D}}\widetilde u^{D,*}_f(x,t)\mu({\rm d}x),$$
where $\widetilde{u}^{D,*}_f(x,t)=\widetilde u^D_{f+w}(x,t)-w(x)$ and
it is the unique solution of
\begin{equation}\label{int-D*}\widetilde u^{D,*}_f(x,t)=\Pi_x[f(\xi_{t\wedge\tau^D})]-\Pi_x\left[\int^{t\wedge\tau_D}_0\psi^*(\xi_s, \widetilde u^{D,*}_f(\xi_s, t-s)){\rm d}s\right],\quad x\in D,\end{equation}
where $\psi^{*}(x,\lambda)=\psi(x,\lambda+w(x))-\psi(x, w(x))$, restricted to $D$, is a branching mechanism of the kind described in the introduction and for each $\mu\in\mathcal{M}_F({ E})$, $(\widetilde X, \mathbb{P}^*_\mu)$ is a superprocess. Specifically, on $ E$,
\begin{equation}
\psi^*(x,\lambda) = -\alpha^*(x)\lambda +\beta(x)\lambda^2+\int_{(0,\infty)}(e^{-\lambda z}-1+\lambda z)\pi^*(x,{\rm d}z),
\label{1}
\end{equation}
where
\[
\alpha^*(x) =\alpha(x) - 2\beta(x)w(x) - \int_{(0,\infty)} (1-e^{-w(x)z})z\pi(x, {\rm d}z)
\]
and
\[
\pi^*(x, {\rm d}z) = e^{-w(x)z}\pi(x,{\rm d}z)\,\,\text{ on }{ E}\times(0,\infty).
\]

\end{lemma}

\begin{proof}
For all $f\in {\rm bp}(\partial ( D\times[0,t) ) )$ we have
\begin{align}
\mathbb{ E}_{\mu}^*(e^{-\langle f,\widetilde X^D_t\rangle})&=\mathbb{ E}_{\mu}(e^{-\langle f,\widetilde X^D_t\rangle}|\mathcal{ E})\notag\\
&=e^{\langle w,\mu\rangle}\mathbb{ E}_{\mu}(e^{-\langle f,\widetilde X^D_t\rangle}1_{\mathcal{ E}})\notag\\
&=e^{\langle w,\mu\rangle}\mathbb{ E}_{\mu}(e^{-\langle f,\widetilde X^D_t\rangle}\mathbb{ E}_{\widetilde X^D_t}(1_{\mathcal{ E}}))\notag\\
&=e^{\langle w,\mu\rangle}\mathbb{ E}_{\mu}(e^{-\langle f+w,\widetilde X^D_t\rangle})\notag\\
&=e^{-\langle \widetilde u^D_{f+w}(\cdot,t)-w,\mu\rangle}.\notag
\end{align}
Using (\ref{int-D}) and (\ref{int-w-D}) then it is straightforward to check that $\widetilde u_f^{D,*}(x,t)=\widetilde u^D_{f+w}(x,t)-w(x)$ is a non-negative
solution to \eqref{int-D*}, which is necessarily unique.
The proof is complete as soon as we can show that  $\psi^*(x, \lambda)$, restricted to $D$, is a branching mechanism which falls into the appropriate class. One easily verifies the formula (\ref{1}) and that the new parameters $\alpha^*$ and $\pi^*$, restricted to $D$, respect the  properties stipulated in the definition of a branching mechanism in the introduction.
\hfill$\square$
\end{proof}

\begin{cor}\label{local-Laplace}
 For any bounded open domain $D\subset\subset E$, any function $f\in {\rm bp}(D)$ and any $\mu\in {\cal M}_F({D})$ satisfying $\langle w,\mu\rangle<\infty$,
$$-\log \mathbb{ E}^*_{\mu}\left(e^{-\langle f, X^D_t\rangle}\right)=\int_{{D}}u^{D,*}_f(x,t)\mu({\rm d}x),$$
where $u^{D,*}_f(x,t)=\widetilde u^D_{f+w}(x,t)-w(x)$ and
it is the unique solution of
\begin{equation}u^{D,*}_f(x,t)=\Pi_x[f(\xi_{t});t<\tau^D]-\Pi_x\left[\int^{t\wedge\tau_D}_0\psi^*(\xi_s, u^{D,*}_f(\xi_s, t-s)){\rm d}s\right],\quad x\in D,\end{equation}
where $\psi^{*}$ is defined by \eqref{1}.

\end{cor}

\subsection{Excursion measure}

Associated to the law of the processes $X$,  are the  measures $\{\mathbb{N}^{*}_x: x\in { E}\}$, defined on the same measurable space as the probabilities $\{\mathbb{P}^*_{\delta_x}:x\in { E}\}$ are defined on, and  which satisfy
\begin{equation}
\mathbb{N}^{*}_x (1- e^{-\langle f, X_t \rangle}) = -\log \mathbb{ E}^{*}_{\delta_x}(e^{-\langle  f, X_t\rangle}) = u^{*}_f(x,t),
\label{DK}
\end{equation}
for all $f\in {\rm bp}( E)$
and $t\geq 0$. Intuitively speaking, the branching property implies that $\mathbb{P}^*_{\delta_x}$ is an infinitely divisible measure on the path space of $X$, that is to say the space of measure-valued cadlag functions, $ \mathbb{D}([0,\infty)\times \mathcal{M}({ E}))$, and  (\ref{DK})  is a `L\'evy-Khinchine' formula in which  $\mathbb{N}^*_x$ plays the role of its `L\'evy measure'. Such measures are formally defined and explored in detail in \cite{DK}.

Note that, by the monotonicity property, for any two open bounded domains, $D_1\subset\subset D_2\subset\subset E$,
\[
\langle f,  X^{D_1}_t\rangle \leq \langle f,  X^{D_2}_t\rangle\qquad \mathbb{N}^*_x\text{-a.e.},
\]
for all $f\in {\rm bp}(D_1)$ understood in the sense of (\ref{lid}), $x\in D_1$ and $t\geq 0$.
Moreover, for an open bounded domain $D$ and $f$ as before, it is also clear that
$\mathbb{N}^*(1- e^{-\langle f, X^D_t \rangle}) =u^{D,*}_f(x,t)$.

The measures $\{\mathbb{N}^{*}_x:x\in{ E}\}$ will play a crucial role in the forthcoming analysis in order to describe the `rate' of a Poisson point process of immigration.

\subsection{A Markov branching process}

In this section we introduce a particular Markov branching process which is built from the components of the $(\psi,\mathcal{P})$-superprocess and which plays a central role in the backbone decomposition.

Recall that we abuse our notation and extend the domain of $w$ with the implicit understanding that $w(\dagger) = 0$. Note, moreover, that thanks to (\ref{localmg}), we have that, for $x\in { E}$, $w(x)^{-1}w(\xi_{t})\exp\left\{-\int_0^{t}\psi(\xi_s,w(\xi_s))/w(\xi_s){\rm d}s\right\}$ is in general a  positive local martingale and hence a supermartingale.
For each $t\geq 0$, let $\mathcal{F}^\xi_t = \sigma(\xi_s: s\leq t)$. Let $\zeta = \inf\{t>0 : \xi_t \in\{\dagger\}\}$ be the life time of $\xi$.
The formula
\begin{equation}\label{cm}
\left.\frac{{\rm d}\Pi_x^{w}}{{\rm d}\Pi_x}\right|_{\mathcal{F}^\xi_t}=\frac{w(\xi_{t})}{w(x)}\exp\left\{-\int_0^{t}\frac{\psi(\xi_s,w(\xi_s))}{w(\xi_s)}{\rm d}s\right\} \qquad \mbox{on } \{t<\zeta\},\quad   t\geq 0, x\in { E},
\end{equation}
uniquely determines a family of (sub-)probability measures $\{\Pi^w_x : x\in { E}\}$.  It is known
that under these new probabilities, $\xi$ is a right Markov process on $ E$; see \cite{Sharpe},
Section 62. We will denote by $\mathcal{P}^{w}$ the semi-group of the $ E\cup\{\dagger\}$-valued process $\xi$ whose probabilities are $\{\Pi_x^{w}:x\in { E}\}$.

\begin{remark}\rm
The equation (\ref{int-w-D}) may  formally be associated with the equation  $Lw(x) -\psi(x,w(x)) = 0$ on $ E$, and the semi-group $\mathcal{P}^w$ corresponds to the diffusion with generator
$$L^w_0:= L^w - w^{-1}Lw  = L^w - w^{-1}\psi(\cdot, w),$$ where
$
L^w u ={w}^{-1} L(wu)
$ for any $u$ in the domain of $L$.  Intuitively speaking, this means that the dynamics associated to $\mathcal{P}^w$, encourages the motion of $\xi$ to visit domains where the global survival rate is high and discourages it from visiting domains where the global survival rate is low. (Recall from (\ref{assumeforeextinguishing}) that larger values of $w(x)$ make extinction of the $(\psi, \mathcal{P})$-superprocess less likely under $\mathbb{P}_{\delta_x}$.)
\end{remark}

Henceforth the process $Z= \{Z_t : t\geq 0\}$ will denote the Markov branching process whose particles move with associated   semi-group $\mathcal{P}^w$. Moreover, the branching generator is given by
\begin{equation}\label{gb}
F(x,s)=q(x)\sum_{n\geq0}p_n(x)(s^n-s),
\end{equation}
where
\begin{equation}
\label{qdef}
q(x)=\psi'(x,w(x))-\frac{\psi(x,w(x))}{w(x)},
\end{equation}
 $p_0(x)=p_1(x)=0$ and for $n\geq2$,
\begin{equation}\label{pl}
p_n(x)=\frac{1}{w(x)q(x)}\left\{\beta(x)w^2(x)1_{\{n=2\}}+w^n(x)\int_{(0,\infty)}\frac{y^n}{n!}e^{-w(x)y}\pi(x,{\rm d}y)\right\}.\notag
\end{equation}
Here we use the notation
\[
\psi'(x,w(x)): = \left.\frac{\partial}{\partial\lambda}\psi(x, \lambda)\right|_{\lambda = w(x)}, \qquad x\in{ E}.
\]
Note that the choice of $q(x)$ ensures that $\{p_n(x): n\geq 0\}$ is a probability mass function. In order to see that $q(x)\geq 0$ for all $x\in { E}$ (but $q\neq 0$), write
\begin{align}
q(x)=\beta(x)w(x)+\frac{1}{w(x)}\int_{(0,\infty)}(1-e^{-w(x)z}(1+w(x)z))\pi(x,{\rm d}z)
\end{align}
and note that $\beta\ge 0$, $w>0$ and
$1-e^{-\lambda z}(1+\lambda z)$, $\lambda\geq 0$, are all non-negative.

\begin{definition}\rm\label{Laplace-2}
In the sequel we shall refer to $Z$ as the  $(\mathcal{P}^w, F)$-backbone. Moreover, in the spirit of Definition \ref{Laplace-1}, for all bounded domains $D$ and $t\geq 0$, we shall also define $\widetilde Z^D_t$ to be the atomic measure,
supported on $\partial(D\times [0,t))$,
describing particles in $Z$ which are first in their genealogical line of descent to exit the domain $D\times [0,t)$.
\end{definition}

Just as with the case of exit measures for superprocesses, we define the random measure, $Z^D = \{Z_t^D: t\geq 0\}$, on $D$ such that
$\langle  f,  Z^D_t\rangle=\langle f, \widetilde Z^{D}_t\rangle$ for any $f\in {\rm bp}(D)$, where we remind the reader that we regard $f$
as a function defined on
${ E}\times [0,\infty)$ as in (\ref{lid}).
As a process in time, $Z^D$ is a Markov branching process,
with branching generator which  is the same as in (\ref{gb}) except that the
 branching rate $q(x)$ is replaced by $q^D(x): = q(x)\mathbf{1}_{D}(x)$,  and
associated motion semi-group given by that of the process $\xi$  killed upon leaving $D$. Similarly to the case of superprocesses, for any two open bounded domains, $D_1\subset\subset D_2\subset\subset E$, the processes $\widetilde{Z}^{D_1}$ and $\widetilde{Z}^{D_2}$
(and hence $Z^{D_1}$ and $Z^{D_2}$) are consistent in the sense that $$\widetilde Z^{D_1}_t= (\widetilde{\widetilde Z_t^{D_2}})^{D_1}$$ for all $t\geq 0$ (and similarly $Z_t^{D_1} = (Z_t^{D_2})^{D_1}$ for all $t\geq 0$).

\section{Local backbone decomposition}\label{localbb}

We are interested in immigrating $(\mathcal{P},\psi^*)$-superprocesses onto the path of an $({\mathcal P}^w, F)$-backbone within the confines of an open, bounded domain $D\subset\subset E$  and initial configuration $\nu\in\mathcal{M}_a(D)$, the space of finite atomic measures in ${D}$ of the form $\sum_{i=1}^n\delta_{x_i}$, where $n\in\mathbb{N}\cup\{0\}$ and $x_1,\cdots,x_n\in D$. There will be three types of immigration: continuous, discontinuous and branch-point immigration which we now describe in detail. In doing so, we shall need to refer to individuals in the process $Z$ for which we shall  use classical Ulam-Harris notation, see for example p290 of Harris and Hardy \cite{HH}. Although the Ulam-Harris labelling of individuals is rich enough to  encode genealogical order, the only feature we really need of the  Ulam-Harris notation is that individuals are uniquely identifiable amongst   $\mathcal{T}$, the set labels of individuals realised in $Z$. For each individual $u\in \mathcal{T}$ we shall write ${b_u}$ and ${d_u}$ for its birth and death times respectively,  $\{z_u(r): r\in[{b_u}, {d_u}]\}$ for its spatial trajectory and $N_u$ for the number of offspring it has at time ${d_u}$. We shall also write $\mathcal{T}^D$ for the set of labels of individuals realised in $Z^D$.
 For each $u\in {\cal T}^D$ we shall also define $$\tau^D_u=\inf\{s\in[{b_u},d_u], z_u(s)\in D^c\},$$
 with the usual convention that $\inf\emptyset : = \infty$.

\begin{definition}\rm
 For $\nu\in\mathcal{M}_a(D)$ and $\mu\in\mathcal{M}_F(D)$, let $Z^D$ be a Markov branching process
 with initial configuration $\nu$, branching generator which  is the same as in (\ref{gb}), except that the
 branching rate $q(x)$ is replaced by $q^D(x): = q(x)\mathbf{1}_{D}(x)$, and
associated motion semi-group given by that of $\mathcal{P}^w$  killed upon leaving $D$. Let  $X^{D,*}$ be  an independent copy of $X^D$ under $\mathbb{P}^*_\mu$. Then we define the measure valued stochastic process $\Delta^D  =\{\Delta^D_t: t\geq 0\}$
such that, for $t\geq 0$,
\begin{equation}\label{Lambda}
\Delta^{D}_t =X^{D,*}_t  +  I^{D,\mathbb{N}^*}_t +  I^{D,\mathbb{P}^*}_t+I^{D,\eta}_t,
\end{equation}
where $I^{D,\mathbb{N}^*} =\{I^{D, \mathbb{N}^*}_t :t\geq 0\}$, $I^{D,\mathbb{P}^*} =\{I^{D,\mathbb{P}^*}_t : t\geq 0\}$ and $I^{D,\eta}=\{I^{D,\eta}_t: t\geq 0\}$ are defined as follows.
\begin{itemize}
\item[i)]({\bf Continuum immigration:})  The process $I^{D,\mathbb{N}^*}$ is measure-valued on
$D$
such that
    $$I^{D,\mathbb{N}^*}_t=\sum_{u\in{\cal T}^D}\sum_{{b_u}< r\le t\wedge{d_u}\wedge\tau^D_u}X_{t-r}^{(D,1,u,r)},$$
   where, given $Z^D$, independently for each $u\in{\cal T}^D$ such that ${b_u}<t$, the processes $X^{(D,1,u,r)}$  are independent copies of the canonical process $X^{D}$, immigrated along the space-time trajectory $\{(z_u(r), r): r\in({b_u}, t\wedge{d_u}\wedge\tau^D_u]\}$ with rate
   \[
   {\rm d}r\times 2\beta(z_u(r)){\rm d}\mathbb{N}^*_{z_u(r)}.
   \]

\item[ii)] ({\bf Discontinuous immigration:}) The process $I^{D, \mathbb{P}^*}$ is measure-valued on
    $D$
    such that
   $$I^{D,\mathbb{P}^*}_t=\sum_{u\in{\cal T}^D}\sum_{{b_u}< r\le t\wedge{d_u}\wedge\tau^D_u}X^{(D,2,u,r)},$$
   where, given $Z^D$, independently for each $u\in{\cal T}^D$ such that ${b_u}<t$, the processes $X^{(D, 2, u,r)}$ are independent copies of the canonical process  $X^{D}$, immigrated along the space-time trajectory $\{(z_u(r), r): r\in({b_u}, t\wedge{d_u}\wedge\tau^D_u]\}$ with rate
   \[
   {\rm d}r\times\int_{y\in(0,\infty)}ye^{-w(z_u(r))y}\pi(z_u(r), {\rm d}y)\times {\rm d}\mathbb{P}^*_{y\delta_{z_u(r)}}.
   \]

\item[iii)]({\bf Branch point biased immigration:}) The process $I^{D,\eta}$ is measure-valued on
$D$
such that
$$I^{D,\eta}_t=\sum_{u\in{\cal T}^D}\mathbf{1}_{\{{d_u}\le t\wedge \tau^D_u\}}X^{(D,3,u)}_{t-{d_u}},$$
where, given $Z^D$, independently for each $u\in{\cal T}^D$ such that ${d_u}<t\wedge\tau^D_u$, the processes $X^{(D, 3, u)}$   are independent copies of the canonical process $X^D$ issued at time ${d_u}$ with law $\mathbb{P}^*_{Y_u\delta_{z_u({d_u})}}$ such that, given $u$ has $n\geq 2$ offspring, the independent random variable $Y_u$ has distribution $\eta_n(z_u(r), {\rm d}y)$, where
\begin{equation}
\eta_n(x, {\rm d}y) =\frac{1}{q(x)w(x)p_n(x)}\left\{
\beta(x)w^2(x)\delta_0({\rm d}y)\mathbf{1}_{\{n = 2\}}+ w(x)^n\frac{y^n}{n!}e^{-w(x)y}\pi(x,{\rm d}y)
\right\}.
\label{il}
\end{equation}

\end{itemize}

It is not difficult to see that $\Delta^D$ is consistent in the domain $D$ in the sense of (\ref{consistent}). Accordingly we denote by  $\mathbf{P}_{(\mu,\nu)}$ the law induced by  $\{\Delta^D_t, D\in {\cal O}( E), t\ge 0\}$, where ${\cal O}( E)$ is the collection of bounded open sets in $ E$.
\end{definition}

The so-called backbone decomposition of $(X^D, \mathbb{P}_\mu)$ for $\mu\in\mathcal{M}_F(D)$ entails looking at the process $\Delta^D$ in the special case that we  randomise the law ${\bf P}_{(\mu,\nu)}$ by replacing the deterministic choice of $\nu$ with a Poisson random measure having intensity measure $w(x)\mu({\rm d}x)$. We denote the resulting law by ${\bf P}_{\mu}$.

\begin{teo}\label{BBth} For any $\mu\in{\cal M}_F({D})$, the process $(\Delta^D, {\bf P}_\mu)$ is Markovian and has the same law as $(X^D, \mathbb{P}_{\mu})$.
\end{teo}

\section{Proof of Theorem \ref{BBth}}\label{5}

The proof involves several intermediary results in the spirit of the non-spatially dependent case of Berestycki et al. \cite{BKMS}. Throughout we take $D$ as an open bounded domain in $ E$.
Any function $f$ defined on $D$ will be extend to
$ E$ by defining $f=0$ on $E\setminus D$.

\begin{lemma}\label{lemma4}
Suppose that  $\mu\in {\cal M}_F(D)$, $\nu\in {\cal M}_a(D)$,   $t\ge 0$ and
 $f\in {\rm bp}(D)$,
 We have
$${\bf E}_{(\mu,\nu)}\left(e^{-\langle f, I^{D, \mathbb{N}^*}_t+I^{D, \mathbb{P}^*}_t\rangle}|\{Z^D_s: s\le t\}\right)=\exp\left\{-\int^t_0\langle\phi(\cdot, u^{D,*}_f(\cdot, t-s)), Z^D_s\rangle {\rm d}s\right\},$$
where
\begin{equation}
\phi(x, \lambda) = 2\beta(x)\lambda + \int_{(0,\infty)} (1- e^{-\lambda z})z\pi(x,{\rm d}z), \qquad x\in D,\, \lambda\geq 0.
\label{els}
\end{equation}
\end{lemma}

\begin{proof}
We write
\begin{equation}
\langle f, I^{D,\mathbb{N}^*}_t+I^{D,\mathbb{P}^*}_t\rangle=\sum_{u\in\mathcal{T}^D}\sum_{{b_u}<r\leq t\wedge{d_u}\wedge\tau^D_u}\langle f,X_{t-r}^{(D,1,u,r)}\rangle+\sum_{u\in\mathcal{T}^D}\sum_{{b_u}<r\leq t\wedge{d_u}\wedge\tau^D_u}\langle f,X_{t-r}^{(D,2,u,r)}\rangle.\notag
\end{equation}
Hence conditioning on $Z^D$, appealing to the independence of the immigration processes together with Campbell's formula and  that $\mathbb{N}^*_{x}(1-e^{-\langle f,X^D_{s}\rangle})=u_f^{D,*}(x,s)$, we have
\begin{eqnarray}
\lefteqn{\mathbf{ E}_{(\mu,\nu)}(e^{-\langle f, I^{D,\mathbb{N}^*}_t\rangle}|\{Z^D_s:s\leq t\})}&&\notag\\
&&=\exp\left\{-\sum_{u\in\mathcal{T}^D}
2\int_{{b_u}}^{t\wedge{d_u}\wedge\tau^D_u}\beta(z_u(r))\cdot\mathbb{N}^*_{z_u(r)}(1-e^{-\langle f,X^D_{t-r}\rangle}){\rm d}r\right\}\notag\\
&&=\exp\left\{-\sum_{u\in\mathcal{T}^D}2\int_{{b_u}}^{t\wedge{d_u}\wedge\tau^D_u}
\beta(z_u(r))u_f^{D,*}(z_u(r),t-r){\rm d}r\right\}.
\label{i1}
\end{eqnarray}
On the other hand
\begin{align}\label{i2}
\mathbf{ E}_{(\mu,\nu)}&(e^{-\langle f, I^{D,\mathbb{P}^*}_t\rangle}|\{Z^D_s:s\leq t\})\notag\\
&=\exp\left\{-\sum_{u\in\mathcal{T}^D}\int_{{b_u}}^{t\wedge{d_u}\wedge\tau^D_u}\int_0^\infty ye^{-w(z_u(r))y}\pi(z_u(r),{\rm d}y)\mathbb{ E}^*_{y\delta_{z_u(r)}}(1-e^{-\langle f,X^D_{t-r}\rangle}){\rm d}r\right\}\notag\\
&=\exp\left\{-\sum_{u\in\mathcal{T}^D}\int_{{b_u}}^{t\wedge{d_u}\wedge\tau^D_u}\int_0^\infty (1-e^{-u_f^{D,*}(z_u(r),t-r)y})ye^{-w(z_u(r))y}\pi(z_u(r),{\rm d}y){\rm d}r\right\}.
\end{align}
Combining (\ref{i1}) and (\ref{i2}) the desired result follows.
\qed
\end{proof}

\begin{lemma}\label{int-equ} Suppose that the real-valued  function $J(s, x, \lambda)$ defined on $[0, T)\times D\times\R$ satisfies that for any $c>0$ there is a constant $A(c)$ such that
$$|J(s, x, \lambda_1)-J(s, x,\lambda_2)|\le A(c)|\lambda_1-\lambda_2|,$$
for all $s\in[0,T)$, $x\in D$ and $\lambda_1,\lambda_2\in[-c, c]$.
Then for any bounded measurable function $g(s, x)$ on $[0, T)\times D$, the integral equation
$$v(t,x)=g(t,x)+\int^t_0\Pi_x\left[J(t-s, \xi_s, v(t-s, \xi_s)); s<\tau^D\right]{\rm d}x,\quad t\in[0, T),$$
has at most one bounded solution.
\end{lemma}
\begin{proof} Suppose that $v_1$ and $v_2$ are two solutions, then there is a constant $c>0$ such that $-c\le v_1, v_2\le c$ and
$$\|v_1-v_2\|(t)\le A(c) \int^t_0\|v_1-v_2\|(s){\rm d}s,$$
where $\|v_1-v_2\|(t)=\sup_{x\in D}|v_1(t, x)-v_2(t, x)|$, $t\in(0, T).$ It follows from Gronwall's lemma (see, for example, Lemma 1.1 on page 1208 of \cite{D1}) that $\|v_1-v_2\|(t)=0, t\in[0, T)$.

\end{proof}

\begin{lemma}\label{lemma5} Fix $t>0$.
Suppose that
$f, h\in {\rm bp}(D)$
and $g_s(x)$ is jointly measurable in $(x,s)\in D\times [0,t]$ and bounded on finite time horizons of $s$ such that $g_s(x) = 0$ for
$x\in D^c$.
Then for any $\mu\in{\cal M}_F(D)$, $x\in D$ and $t\ge 0$,
$$
e^{- W(x, t)}
:=
{\bf E}_{(\mu,\delta_x)}\left[\exp\left(-\int^t_0\langle g_{t-s}, Z^D_s\rangle {\rm d}s-\langle f, I^{D, \eta}_t\rangle-\langle h, Z^D_t\rangle\right)\right],
$$
where $e^{- W(x, t)}$ is the unique $[0,1]$-valued solution to the integral equation
\begin{eqnarray}\label{int-W}
w(x)e^{-W(x,t)}&=&\displaystyle
{\color{black}\Pi_x\left[w(\xi_{t}) e^{-h(\xi_{t})},t<\tau_D\right]}\notag\\
&&+\Pi_{x}
\Bigg[\int^{t\wedge\tau_{D}}_0[H_{t-s}(\xi_s,-w(\xi_s)e^{-W(\xi_s, t-s)})-w(\xi_s)e^{-W(\xi_s, t-s)}g_{t-s}(\xi_s)\notag\\
&&\hspace{3cm}-\psi(\xi_s, w(\xi_s))e^{-W(\xi_s,t-s)}]{\rm d}s\Bigg],
\end{eqnarray}
for $x\in D$, where
$$H_{t-s}(x,\lambda)=q(x)\lambda +\beta(x)\lambda^2+\int^{\infty}_0(e^{-\lambda y}-1+\lambda y)e^{-(w(x)+u^{D,*}_f(x, t-s))y}\pi(x,{\rm d}y),\quad x\in D,$$
and $q(x)$ was defined in (\ref{qdef}).
\end{lemma}
\begin{proof}
Following Evans and O'Connell \cite{EO} it suffices to prove the result in the case when $g$ is time invariant. To this end, let us start by defining the semi-group $\mathcal{P}^{h, D}$ by
\begin{eqnarray}\label{F-K-semi}
\mathcal{P}^{h,D}_t[k](x)&=&\displaystyle\Pi_x\left(e^{-\int_0^{t\wedge\tau^D}h(\xi_s){\rm d}s}k(\xi_{t\wedge\tau^D})\right)\notag\\
&=&{\color{black}\displaystyle\Pi_x\left(e^{-\int_0^{t}h(\xi_s){\rm d}s}k(\xi_{t}); t<\tau^D\right),}
\quad\text{ for }h,k\in {\rm bp}({\color{black} D}),
\end{eqnarray}
where, for convenience, we shall write
\begin{equation}\label{semi-w-D}\mathcal{P}^{D}_t[k]=\mathcal{P}^{0,D}_t[k].
\end{equation}

Define the function $\chi(x) =
\psi(x, w(x))/w(x)$. Conditioning on the first splitting time in the process $Z^D$ and recalling that the branching occurs at the spatial rate
$q^D(x)=\mathbf{1}_D(x)( \psi'(x, w(x))
- \chi(x))$ we get that for any $x\in D$,
\begin{eqnarray}\label{i3}
e^{-W(x,t)}&=&\frac{1}{w(x)}\mathcal{P}^{g+q+\chi,D}_t[we^{-h}](x)\\
&&\hspace{-2cm}+\Pi^w_x\Bigg[\int_0^{t\wedge\tau^D}\exp\left(-\int^s_0(g+q)(\xi_r){\rm d}r\right)\notag\\
&&\left\{q(\xi_s)\sum_{n\geq2}p_n(\xi_s)e^{-nW(\xi_s,t-s)}\int_{(0,\infty)}\eta_n(\xi_s,{\rm d}y)e^{-yu_f^{D,*}(\xi_s,t-s)}\right\}{\rm d}s\Bigg].
\end{eqnarray}
From (\ref{il}) we quickly find that for $x\in D$
\begin{align}
\sum_{n\geq2}&p_n(x)e^{-nW(x,t-s)}\int_{(0,\infty)}\eta_n(x,{\rm d}y)e^{-yu_f^{D,*}(x,t-s)}\notag\\
&=\frac{1}{q(x)w(x)}\left\{H_{t-s}(x,-w(x)e^{-W(x,t-s)})+w(x)q(x)e^{-W(x,t-s)}\right\}.\notag
\end{align}
Using the above expression in (\ref{i3}) we have that
\begin{eqnarray*}
w(x)e^{-W(x,t)}&=&\mathcal{P}^{g+q+\chi,D}_t[w e^{-h}](x)\\
&&\hspace{-2cm}+\Pi_x\Bigg[\int_0^{t\wedge\tau^D}\exp\left(-\int^s_0(g+q+\chi)(\xi_r){\rm d}r\right)\\
&&
\left[(H_{t-s}(\xi_s,-w(\xi_s)e^{-W(\xi_s,t-s)})+w(\xi_s)q(\xi_s)e^{-W(\xi_s,t-s)})\right]{\rm d}s\Bigg].
\end{eqnarray*}
 Now appealing to Lemma 1.2 in Dynkin \cite{D2}
 and recalling that $ \chi(\cdot) = \psi(\cdot,w(
 \cdot))/w(\cdot)$  on $D$, we may deduce that for any $x\in D$,
\begin{eqnarray}
\hspace{-2cm}w(x)e^{-W(x,t)}&=&\mathcal{P}_t^{D}[we^{-h}](x)\notag\\
&&\hspace{-2cm}+\Pi_x\Bigg[\int_0^{t\wedge\tau^D}
[H_{t-s}(\xi_s,-w(\xi_s)e^{-W(\xi_s,t-s)})-w(\xi_s)g(\xi_s)e^{-W(\xi_s,t-s)}\notag \\
&&\hspace{5cm}- \psi(\xi_s, w(\xi_s)) e^{-W(\xi_s, t-s)}
]{\rm d}s\Bigg]
\label{useunique}
\end{eqnarray}
as required. Note that in the above computations we have implicitly used that $w$ is uniformly bounded away from 0 and $\infty$ on $D$.

To complete the proof we need to show uniqueness of solutions to (\ref{useunique}).  Lemma \ref{int-equ}  offers sufficient conditions for uniqueness of solutions to a general family of integral equations which includes (\ref{useunique}). In order to check these sufficient conditions, let us define $\overline{w}^D = \sup_{y\in D} w(y)$.   Thanks to
 Assumption (A) we have that $0<\overline{w}^D<\infty$.
For $s\geq 0$, $x\in D$ and $\lambda\in[0,\overline{w}^D]$, define the function $J(s, x, \lambda): = [H_s(x,-\lambda)-(g(x)+\chi(x))\lambda ]$.
We rewrite (\ref{useunique}) as
$$
w(x)e^{-W(x,t)}=\mathcal{P}_t^{D}[we^{-h}](x)+\int_0^{t}\Pi_x\left[J(t-s,\xi_s, w(\xi_s)e^{-W(\xi_s,t-s)}); s<\tau^D\right]
{\rm d}s.
$$
Lemma \ref{int-equ}  tells us that (\ref{useunique}) has a unique solution as soon as we can show that $J$ is continuous in $s$ and that for each fixed $T>0$, there exists a $K>0$ (which may depend on $D$ and $T$) such that
\[
\sup_{s\leq T}\sup_{y\in D}|J (s, y, \lambda_1) - J(s, y, \lambda_2)|\leq K |\lambda_1 - \lambda_2|
,\quad \lambda_1, \lambda_2\in (0,\overline{w}^D].\]
Recall that  $g(y)$ is assumed to be bounded,
moreover, Assumption (A) together with the fact that
\begin{equation}
\sup_{y\in D}\left\{ |\alpha(y)| + \beta(y) + \int_{(0,\infty)}( z\wedge z^2) \pi(y,{\rm d}z)\right\}<\infty
\label{bounded}
\end{equation}
also implies that $\chi$ is bounded on $D$.
Appealing to  the triangle inequality, it now suffices to check
that
for each fixed $T>0$, there exists a $K>0$ such that
\begin{equation}
\sup_{s\leq T}\sup_{y\in D}|H_s(y, -\lambda_1) - H_s(y, -\lambda_2)|\leq K|\lambda_1 - \lambda_2|,\quad \lambda_1, \lambda_2\in (0,\overline{w}^D].
\label{check-this}
\end{equation}
First note from Proposition 2.3 of Fitzsimmons \cite{fitz} that
\begin{equation}
\sup_{s\leq T}\sup_{x\in D}u^{D,*}_f(x,s)<\infty.
\label{u*bounded}
\end{equation}
Straightforward differentiation of the function $H_s(x, -\lambda)$ in the variable $\lambda$ yields
\[
-\frac{\partial}{\partial\lambda}H_s(x, - \lambda)  = q(x) -2\beta(x)\lambda + \int_{(0,\infty)} (1-{\rm e}^{\lambda z}){\rm e}^{-(w(x) + u^{D,*}_f(x,s))z}z\pi(x, {\rm d}z).
\]
Appealing to (\ref{bounded}) and (\ref{u*bounded}) it is not difficult to show that the derivative above is uniformly bounded in absolute value for $s\leq T$, $x\in D$ and $\lambda\in[0,\overline{w}^D]$, from which (\ref{check-this}) follows by straightforward linearisation. The proof is now complete.
\qed
\end{proof}

\begin{teo}\label{spine-decomp} For every $\mu\in\mathcal{M}_F(D)$,  $\nu\in\mathcal{M}_a(D)$ and $f, h\in{\rm bp}(D)$ 
\begin{equation}\label{Laplace-Delta-Z}{\bf E}_{(\mu, \nu)}\left(e^{-\langle f, \Delta^{D}_t\rangle-\langle h, Z^{D}_t\rangle}\right)=e^{-\langle u^{D,*}_f(\cdot, t),\mu\rangle-\langle v^D_{f,h}(\cdot,t),\nu\rangle},\end{equation}
where $e^{-v^D_{f,h}(x,t)}$ is the unique $[0,1]$-solution to the integral equation
\begin{eqnarray}\label{int4}
w(x)e^{-v^D_{f,h}(x,t)}&=&
\Pi_x\left[w(\xi_{t}) e^{-h(\xi_{t})}; t<\tau_D\right]\notag\\
&&\hspace{-3.5cm}+\Pi_{x}\left[\int^{t\wedge\tau_{D}}_0[\psi^*(\xi_s,-w(\xi_s)e^{-v^D_{f,h}(\xi_s, t-s)}+u^{D,*}_f(\xi_s, t-s))-\psi^*(\xi_s, u^{D,*}_f(\xi_s, t-s))]{\rm d}s\right].\end{eqnarray}
\end{teo}

\begin{proof} Thanks to Lemma \ref{lemma2}  it suffices to prove that
\begin{equation}
\mathbf{ E}_{(\mu,\nu)}(e^{\langle -f,I^D_t\rangle-\langle h,Z^D_t\rangle})=e^{-\langle v^D_{f,h}(\cdot,t),\nu\rangle},\notag
\end{equation}
where $I^D:=I^{D,\mathbb{N}^*}+I^{D,\mathbb{P}^*}+I^{D,\eta}$, and $v^D_{f,h}$ solves (\ref{int4}). Putting Lemma \ref{lemma4} and Lemma \ref{lemma5}  together we only need  to show that, when $g_{t-s}(\cdot)=\phi(\cdot,u_f^{D,*}(\cdot,t-s))$ (where $\phi$ is given by (\ref{els})), we have that $\exp\{-W(x,t)\}$ is the unique $[0,1]$-valued solution  to (\ref{int4}). Again following the lead of  \cite{BKMS}, in particular referring to Lemma 5 there, it is easy to see that on $D$
\begin{align}
H_{t-s}(\cdot,-w(\cdot)e^{-W(\cdot,t-s)})&-\phi(\cdot, u_f^{D,*}(\cdot,t-s))w(\cdot)e^{-W(\cdot,t-s)}-\frac{\psi(\cdot,w(\cdot))}{w(\cdot)}w(\cdot)e^{-W(\cdot,t-s)}\notag\\
&=\psi^*(\cdot,\-w(\cdot)e^{-W(\cdot,t-s)}+u_f^{D,*}(\cdot,t-s))-\psi^{*}(\cdot,u_f^{D,*}(\cdot,t-s)),\notag
\end{align}which implies that $\exp\{-W(x,t)\}$ is the unique $[0,1]$-valued  solution to (\ref{int4}).\qed
\end{proof}

{\bf Proof of Theorem  \ref{BBth}:}.
 The proof is guided by the calculation in the proof of Theorem 2 of \cite{BKMS}. We start by addressing the claim that $(\Delta^D, {\bf P}_\mu)$ is a Markov process. Given the Markov property of the pair $(\Delta^D, Z^D)$, it suffices to show that, given $\Delta^D_t$, the atomic measure $Z^D_t$ is equal in law to a Poisson random measure with intensity $w(x)\Delta^D_t$. Thanks to Campbell's formula for Poisson random measures, this is equivalent to showing that for all   $h\in{\rm bp}(D)$,
$${\bf E}_{\mu}(e^{-\langle h, Z^D_t\rangle}|\Delta^D_t)=e^{-\langle w\cdot(1- e^{-h}),\Delta^D_t\rangle},$$which in turn is equivalent to showing that for all $f, h\in {\rm bp}(D)$,
\begin{equation}\label{Laplace-Delta-Z2}{\bf E}_{\mu}(e^{-\langle f, \Delta^D_t\rangle-\langle h, Z^D_t\rangle})={\bf E}_{\mu}(e^{-\langle w\cdot(1- e^{-h})+f,\Delta^D_t\rangle}).\end{equation}
Note from \eqref{Laplace-Delta-Z} however that when we randomize $\nu$ so that it has the law of a Poisson random measure with intensity $w(x)\mu({\rm d}x)$, we find the identity
$${\bf E}_{\mu}(e^{-\langle f, \Delta^D_t\rangle-\langle h, Z^D_t\rangle})=\exp\left\langle -u^{D,*}_f(\cdot,t)-w\cdot(1-e^{-v^D_{f,h}(\cdot,t)}),\mu\right\rangle.$$
Moreover, if we replace $f$ by $w\cdot(1-e^{-h})+f$ and $h$ by $0$ in \eqref{Laplace-Delta-Z} and again randomize $\nu$ so that it has the law of a Poisson random measure with intensity $w(x)\mu({\rm d}x)$ then we get
$${\bf E}_{\mu}\left(e^{-\langle w\cdot(1- e^{-h})+f,\Delta^D_t\rangle}\right)=\exp{\left\langle -u^{D, *}_{w\cdot(1-e^{-h})+f}(\cdot, t)-w\cdot\left(1-\exp\left\{-v^D_{w\cdot(1-e^{-h})+f,0}\right\}\right),\mu\right\rangle}.$$
These last two observations indicate that (\ref{Laplace-Delta-Z2}) is equivalent to showing that, for all $f,h$ as stipulated above and $t\ge 0$,
\begin{equation}\label{eq-1}u^{D,*}_f(x,t)+w(x)(1-e^{-v^D_{f,h}(x,t)})=u^{D,*}_{w\cdot(1-e^{-h})+f}(x,t)+w(x)(1-e^{-v^D_{w\cdot(1-e^{-h})+f, 0}(x,t)}).\end{equation}
Note that both left and right-hand side of the equality above are necessarily non-negative
given that they are Laplace exponents of the left and right-hand sides of \eqref{Laplace-Delta-Z2}. Making use
of \eqref{int-D*}, \eqref{int-w-D}, and \eqref{int4}, it is computationally very straightforward to show that both left
and right-hand side of \eqref{eq-1} solve \eqref{int-D} with initial condition $f + w(1-e^{-h})$, which is bounded in $\overline D$. Since (1.2)
has a unique solution with this initial condition, namely $u^D_{f+w\cdot(1-e^{-h})}(x,t)$, we conclude that
\eqref{eq-1} holds true. The proof of the claimed Markov property is thus complete.

Having now established the Markov property, the proof is complete as soon as we can
show that $(\Delta^D, {\bf P}_\mu)$ has the same semi-group as $(X^D, \mathbb P_\mu)$. However, from the previous part of
the proof we have already established that when
$f,h\in {\rm bp}(D)$,
\begin{equation}
{\bf E}_\mu\left(e^{-\langle h, Z^D_t\rangle-\langle f, \Delta^D_t\rangle}\right)=e^{-\langle u^D_{w(1-e^{-h})+f}(\cdot,t),\mu\rangle}=\mathbb E_\mu\left(e^{-\langle f+w(1-e^{-h}), X^D_t\rangle}\right).
\label{poisson-laplace}
\end{equation}
In particular, choosing $h=0$ we find
$${\bf E}_\mu\left(e^{-\langle f,\Delta^D_t\rangle}\right)=\mathbb E_\mu\left(e^{-\langle f, X^D_t\rangle}\right),\qquad t\geq 0,$$
which is equivalent to saying that  the semi-groups of $(\Delta^D, {\bf P}_\mu)$ and $(X^D, \mathbb P_\mu)$ agree.
\qed

\section{Global backbone decomposition}\label{globalbb}

So far we have localized our computations to an open bounded domain $D$. Our ultimate objective is to provide a backbone decomposition on the whole domain $ E$. To this end,
let $D_n$ be a sequence of  open bounded domains in $ E$ such that $D_n\uparrow{ E}$. Let $X^{D_n}$, $\Delta^{D_n}$ and $Z^{D_n}$ be defined as in previous sections with $D$ being replaced by $D_n$.

\begin{lemma}\label{limit-Delta-Z}
For any
$h,f\in {\rm bp}( E)$ with compact support
 and any $\mu\in{\cal M}_F({ E})$, we have that for any $t\ge 0$,each element of the pair
 $\{(\langle h, Z^{D_n}_s\rangle, \langle f, \Delta^{D_n}_s\rangle): s\leq t\}$
 pathwise increases $\mathbf{P}_\mu$-almost surely as $n\to\infty$. The limiting pair of processes,
 here denoted by  $\{(\langle h, Z^{\rm min}_s\rangle, \langle f,\Delta^{\rm min}_s\rangle): s\leq t\}$,
 are  such that $\langle f, \Delta^{\rm min}_t\rangle$ is equal in law to $\langle f, X_t\rangle$ and,
 given $\Delta^{\rm min}_t$, the law of $Z^{\rm min}_t$ is a Poisson random field with intensity
 $w(x)\Delta^{\rm min}_t({\rm d}x)$. Moreover, $Z^{\rm min}$ is a $(\mathcal{P}^w, F)$ branching process
 with branching generator as in (\ref{gb})  and
associated motion semi-group given by that of the process $\xi$.
\end{lemma}

\begin{proof}
Appealing to the stochastic consistency of $Z^D$ and $\Delta^D$ in the domain $D$, it is clear that both $\langle h, Z^{D_n}_t, \rangle$ and $\langle f, \Delta^{D_n}_t \rangle$ are  almost surely increasing in $n$. It therefore follows that the limit as $n\to\infty$ exists for both
$\langle h, Z^{D_n}_t\rangle$ and $\langle f, \Delta^{D_n}_t\rangle$, $\mathbf{P}_\mu$-almost surely. In light of the discussion at the end of the proof of Theorem \ref{BBth}, the distributional properties of the limiting pair are established as soon as we show that
\begin{equation}
-\log \mathbf{ E}_{\mu}\left(e^{-\langle h, Z^{\rm min}_t\rangle - \langle f, \Delta^{\rm min}_t\rangle}\right)=\int_{ E}u_{w(1-e^{-h})+f}(x, t)\mu({\rm d}x),\quad t\ge 0.
\label{poissonisation}
\end{equation}

Assume temporarily  that $\mu$ has compact support so  that there exists an $n_0\in\mathbb{N}$ such that  for $n\geq n_0$ we have that $\mbox{supp}\mu\subset D_n$ and $h=f=0$ on $D^c_n$.
Thanks to (\ref{poisson-laplace}) and monotone convergence (\ref{poissonisation}) holds as soon as we can show  that $u_g^{D_n}\uparrow u_g$ for all $g\in{\rm bp}( E)$ satisfying $g=0$ on $D^c_n$ for $n\ge n_0$.
By \eqref{Laplace-D-r} and \eqref{int-D'}, we know that $u_g^{D_n}(x,t)$ is the unique non-negative solution to the integral equation
\begin{equation}\label{int-Dn-1}
u^{D_n}_g(x,t)=\Pi_x[g(\xi_{t\wedge\tau_{D_n}})] - \Pi_x\left[\int^{t\wedge\tau_{D_n}}_0\psi(\xi_s, u^{D_n}_g(\xi_s, t-s)){\rm d}s\right].
\end{equation}
Using Lemma 1.5 in \cite{D1} we can rewrite the above integral equation in the form
\begin{eqnarray}\label{int-Dn-2}
u^{D_n}_g(x,t)&=& \Pi_x\left[g(\xi_{t\wedge\tau_{D_n}})\exp\left(\int^{t\wedge\tau_{D_n}}_0\alpha(\xi_s){\rm d}s\right)\right]\notag\\
&&\hspace{-2.5cm}-\Pi_x\left[\int^{t\wedge\tau_{D_n}}_0\exp\left(\int^{s}_0\alpha(\xi_r){\rm d}r\right)\left[\psi(\xi_s, u^{D_n}_g(\xi_s, t-s))+\alpha(\xi_s) u^{D_n}_g(\xi_s, t-s))\right]{\rm d}s\right].
\end{eqnarray}
Since $g=0$ on $D^c_n $ for $n\ge n_0$, we have
$$ \Pi_x\left[g(\xi_{t\wedge\tau_{D_n}})\exp\left(\int^{t\wedge\tau_{D_n}}_0\alpha(\xi_s){\rm d}s\right)\right]
= \Pi_x\left[g(\xi_{t})\exp\left(\int^{t}_0\alpha(\xi_s){\rm d}s\right);t<\tau_{D_n}\right],$$
which is increasing in $n$. By the comparison principle, $u^{D_n}_g$ is increasing in $n$ (see Theorem 3.2 in part II of  \cite{D1}).
Put $\tilde u_g=\lim_{n\to\infty} u^{D_n}_g$.  Note that $\psi(x,\lambda)+\alpha(x)\lambda$ is
increasing in $\lambda$. Letting $n\to\infty$ in \eqref{int-Dn-2}, by the monotone convergence theorem,
$$
 \tilde u_g(x,t)=\displaystyle \mathcal{P}^\alpha_sg(x)- \Pi_x\int^{t}_0\mathcal{P}^\alpha_s\left[\psi(\cdot, \tilde u_g(\cdot, t-s))+\alpha(\cdot)\tilde u_g(\cdot, t-s))\right]{\rm d}s,
 $$
where
$$
\mathcal{P}^\alpha_tf=\Pi_x\left[g(\xi_{t})\exp\left(\int^{t}_0\alpha(\xi_s){\rm d}s\right)\right],\quad g\in {\rm bp}( E),
$$
which in turn is equivalent to
$$ \tilde u_g(x,t)= \mathcal{P}_sg(x)-\Pi_x\int^{t}_0\mathcal{P}_s\psi(\cdot, \tilde u_g(\cdot, t-s)){\rm d}s.$$
Therefore, $\tilde u_g$ and $u_g$ are two solutions of \eqref{u_f} and hence by uniqueness they are the same, as required.

To remove the initial assumption that $\mu$ is compactly supported, suppose that $\mu_n$ is a sequence of compactly supported measures with mutually disjoint support  such that $\mu = \sum_{k\geq 1}\mu_k$. By considering (\ref{poissonisation}) for $\sum_{k = 1}^n\mu_k$  and taking limits as $n\uparrow\infty$ we  see that (\ref{poissonisation})  holds for $\mu$. Note in particular that the limit on the left hand side of (\ref{poissonisation}) holds as a result of the additive property of the backbone decomposition in the initial state $\mu$. \qed
\end{proof}

Note that, in the style of the proof given above (appealing to monotonicity and the maximality principle) we can easily show that the processes $X^{D_n,*}$, $n\geq 1$, converge distributionally at fixed times, and hence in law, to the process $(X,\mathbb{P}^*_\mu)$; that is, a $(\mathcal{P}, \psi^*)$-superprocess.
With this in mind, again appealing to the consistency and monotonicity of the local backbone decomposition in the size of domain the following, our main result, follows
as a simple corollary of Lemma \ref{limit-Delta-Z}.
\begin{cor}
 Suppose that  $\mu\in\mathcal{M}_F( E)$.
  Let $Z$ be a $(\mathcal{P}^w, F)$-Markov branching process with initial configuration consisting of a Poisson random field of particles in $ E$ with intensity $w(x)\mu({\rm d}x)$. Let $X^{*}$ be an independent copy of $(X,\mathbb{P}^*_\mu)$. Then define the measure valued stochastic process $\Delta  =\{\Delta_t: t\geq 0\}$
such that, for $t\geq 0$,
\begin{equation}\label{Lambda}
\Delta_t =X^{*}_t  +  I^{\mathbb{N}^*}_t +  I^{\mathbb{P}^*}_t+I^{\eta}_t,
\end{equation}
where $I^{\mathbb{N}^*} =\{I^{\mathbb{N}^*}_t :t\geq 0\}$, $I^{\mathbb{P}^*} =\{I^{\mathbb{P}^*}_t : t\geq 0\}$ and $I^{\eta}=\{I^{\eta}_t: t\geq 0\}$ are defined as follows.
\begin{itemize}
\item[i)]({\bf Continuum immigration:})  The process $I^{\mathbb{N}^*}$ is measure-valued on
$ E$
such that
    $$I^{\mathbb{N}^*}_t=\sum_{u\in{\cal T}}\sum_{{b_u}< r\le t\wedge{d_u}}X_{t-r}^{(1,u,r)},$$
   where, given $Z$, independently for each $u\in{\cal T}$ such that ${b_u}<t$, the processes $X^{(1,u,r)}$  are independent copies of the canonical process $X$, immigrated along the space-time trajectory $\{(z_u(r), r): r\in({b_u}, t\wedge{d_u}]\}$ with rate
   \[
   {\rm d}r\times 2\beta(z_u(r)){\rm d}\mathbb{N}^*_{z_u(r)}.
   \]

\item[ii)] ({\bf Discontinuous immigration:}) The process $I^{ \mathbb{P}^*}$ is measure-valued on
    $ E$ such that
   $$I^{\mathbb{P}^*}_t=\sum_{u\in{\cal T}}\sum_{{b_u}< r\le t\wedge{d_u}}X^{(2,u,r)},$$
   where, given $Z$, independently for each $u\in{\cal T}$ such that ${b_u}<t$, the processes $X^{( 2, u,r)}$ are independent copies of the canonical process $X$, immigrated along the space-time trajectory $\{(z_u(r), r): r\in({b_u}, t\wedge{d_u}]\}$ with rate
   \[
   {\rm d}r\times\int_{y\in(0,\infty)}ye^{-w(z_u(r))y}\pi(z_u(r), {\rm d}y)\times {\rm d}\mathbb{P}^*_{y\delta_{z_u(r)}}.
   \]

\item[iii)]({\bf Branch point biased immigration:}) The process $I^{\eta}$ is measure-valued on
$ E$
such that
$$I^{\eta}_t=\sum_{u\in{\cal T}^D}\mathbf{1}_{\{{d_u}\le t\}}X^{(3,u)}_{t-{d_u}},$$
where, given $Z$, independently for each $u\in{\cal T}$ such that ${d_u}<t$, the processes $X^{(3, u)}$   are independent copies of the canonical process $X$ issued at time ${d_u}$ with law $\mathbb{P}^*_{Y_u\delta_{z_u({d_u})}}$ such that, given $u$ has $n\geq 2$ offspring, the independent random variable $Y_u$ has distribution $\eta_n(z_u(r), {\rm d}y)$, where $\eta_n(x, {\rm d}y)$ is defined by (\ref{il}).

\end{itemize}

\noindent Then $(\Delta, {\bf P}_\mu)$ is Markovian and has the same law as $(X, \mathbb{P}_{\mu})$.

\end{cor}

\section*{Acknowledgements}
We would like to thank Maren Eckhoff for a number of helpful comments on earlier versions of this paper. Part of this research was carried out whilst AEK was on sabbatical at ETH Z\"urich, hosted by the Forschungsinstitute f\"ur Mathematik, for whose hospitality he is grateful. The research of YXR is supported in part by the NNSF of China (Grant Nos. 11271030 and 11128101).

\end{document}